\documentclass[12pt]{amsart}

\usepackage{MnSymbol}
\newtheorem{theorem}{Theorem}[section]
\newtheorem{lemma}[theorem]{Lemma}
\newtheorem{prop}[theorem]{Proposition}

\theoremstyle{definition}

\theoremstyle{remark}
\newtheorem{remark}[theorem]{Remark}

\newcommand{\rr}{{\mathbb R}}

\newcommand{\N}{{\mathbb N}}

\newcommand{\E}{\mathbb{E}}

\usepackage{comment}
\usepackage{tikz}
\usepackage{float}
\usepackage{relsize}
\usepackage{enumitem}
\usepackage[T1]{fontenc}
\usepackage[english]{babel}
\usepackage{color}
\usepackage{graphicx}
\usepackage{cite}
\usepackage{bbm}

\begin{document}
	
	\sloppy
	
	\date{\today}
	
	\title[Tanaka formula for fractional SDEs]{Tanaka formula for SDEs driven by fractional Brownian motion} 
	
	\author{Tommi Sottinen}
	\address{Tommi Sottinen, School of Technology and Innovations, University of Vaasa, P.O.Box 700, FI-65101 Vaasa, Finland}
	\email{tommi.sottinen\@@{}uwasa.fi}
	
	\author{Ercan S\"onmez}
	\address{Ercan S\"onmez, Ruhr University Bochum, Faculty of Mathematics, Bochum, Germany}
	\email{ercan.soenmez\@@{}rub.de}

	\author{Lauri Viitasaari}
	\address{Lauri Viitasaari, Department of Information and Service Management, Aalto University School of Business, Finland}
	\email{lauri.viitasaari\@@{}aalto.fi}

	\begin{abstract}
		We derive a Tanaka-type formula for the solution of a stochastic differential equation driven by fractional Brownian motion with Hurst parameter \( H > \frac{1}{2} \). While Tanaka formulas for the fractional Brownian motion itself have been established, a corresponding result for non-linear SDEs driven by fBm has so far been unavailable. Our formula reveals a structure not previously observed: it features both a Skorokhod integral and a Malliavin trace correction, where the analogue of the local time appears through a double integral involving the Dirac distribution and the Malliavin derivative of the solution. A second double integral captures the variation of the diffusion coefficient along the flow. A key step in our analysis is a novel method to establish \( L^2 \)-convergence of the trace term, which avoids the use of white noise calculus and instead exploits Gaussian-type density estimates for the law of the solution. The result applies to a broad class of equations under suitable regularity assumptions and extends naturally to convex functionals. As special cases, we recover known identities for the fractional Brownian motion and the fractional Ornstein--Uhlenbeck process.
	\end{abstract}
	
	\maketitle
	
	\allowdisplaybreaks

	\section{Introduction}
	
	The fractional Brownian motion (fBm) is a family of centered Gaussian processes \( (B_t^H)_{t \in [0,T]} \) indexed by the Hurst parameter \( H \in (0,1) \). Unlike standard Brownian motion, fBm exhibits long-range dependence and non-Markovian dynamics for \( H \neq \frac{1}{2} \), and is not a semimartingale unless \( H = \frac{1}{2} \). These features make it an attractive model in various applications, but they also preclude the use of classical Itô calculus. Over the last two decades, several stochastic calculus frameworks have been developed to handle fBm, including the Malliavin--Skorokhod approach (see e.g.~\cite{NualartBook}) and pathwise integration in the sense of Young or Riemann--Stieltjes for \( H > \frac{1}{2} \) (see e.g. \cite{za}).
	
	An important class of questions in stochastic analysis concerns the behavior of non-smooth functionals of stochastic processes, such as \( |X_t - x| \), or more generally, convex functions of the state. For Brownian motion, the classical Tanaka formula provides a precise decomposition of such functionals into a martingale, a drift, and a local time term. Extensions of this formula to the setting of fBm have been established using Malliavin calculus and Skorokhod integrals, notably in the work of Coutin, Nualart, and Tudor~\cite{coutin}, where a Tanaka-type identity is obtained for fBm involving an explicit chaos expansion of the local time.
	
	In many contexts, however, the process of interest is not fBm itself, but the solution of a stochastic differential equation (SDE) driven by fBm. The study of such SDEs has received considerable attention (see e.g.~\cite{baudoin, NR, ssv}), and several works have extended Itô-type formulas to this setting, both in the Malliavin calculus framework and via pathwise techniques. Nevertheless, a general Tanaka formula for solutions of fBm-driven SDEs is missing from the literature.
	
	The present work aims to address this gap. Our goal is to derive a Tanaka-type formula for solutions \( (X_t)_{t \in [0,T]} \) to SDEs of the form
	\[
	dX_t = b(X_t) \, dt + \sigma(X_t) \, dB_t^H,
	\]
	under suitable conditions related to the coefficient functions and with \( H > \frac{1}{2} \). The resulting identity reveals new structural features not captured by existing results for fBm itself and blends techniques from Malliavin calculus and pathwise analysis. In particular, the analogue of the local time term arises through a Malliavin trace correction involving the Malliavin derivative of the solution.
	
	Existing results have established Tanaka-type identities for the fractional Brownian motion itself, see \cite{coutin}. Extensions to convex functionals and related Itô-type identities have been considered in various works (see e.g.~\cite{Bender2003, Gradinaru2003}), primarily in settings where the underlying process is fBm or a smooth functional thereof. These formulas do not extend to more general processes. In contrast, the case of nonlinear stochastic differential equations driven by fBm has not previously been treated in the context of Tanaka-type formulas. The presence of nontrivial coefficients combined with the roughness and non-Markovian nature of the driving noise presents new challenges that go beyond the linear case.
	
	Our main result fills this gap. We establish a Tanaka-type identity for \( |X_t - x| \), which includes a drift term, a Skorokhod integral, and a Malliavin trace correction. The formula reveals a structure not previously observed: in particular, the analogue of the local time appears through a double integral involving the Dirac distribution and the Malliavin derivative of the solution, alongside an additional correction term reflecting the variation of the diffusion coefficient along the flow of the process. This representation arises from the decomposition of the pathwise integral into a divergence integral and a trace term, and captures features specific to the non-semimartingale setting. We provide natural assumptions under which one obtains the Tanaka formula. 
	Our result applies to a broad class of equations, including those with non-Gaussian solutions, and furthermore recovers earlier identities for fractional Brownian motion and the fractional Ornstein--Uhlenbeck process as special cases.
	
	%We illustrate the applicability of our results by recovering earlier identities for fractional Brownian motion and the fractional Ornstein--Uhlenbeck process as special cases. On top of that, our assumptions can be easily verified for arbitrary equations with smooth and bounded coefficients under more restrictive assumption $H>\frac34$. 
	
	%Our result applies to a broad class of equations, including those with smooth, bounded coefficients, and furthermore recovers earlier identities for fractional Brownian motion and the fractional Ornstein--Uhlenbeck process as special cases.
	
	The main result is stated in Theorem~\ref{main} below. It applies under general assumptions on the coefficients \( b \) and \( \sigma \), sufficient to ensure Malliavin differentiability of the solution. A key challenge in the proof arises in establishing \( L^2 \)-convergence of the Malliavin trace term. Existing approaches rely on white noise analysis, which requires advanced tools from white noise calculus, powerful but not always readily accessible (see e.g. \cite{biagini2004introduction,hu2005weighted}). Instead, we develop an alternative approach based on the fact that the law of the solution admits a smooth density with Gaussian-type bounds, a property known to hold for equations with regular coefficients. This enables us to control the singular integrals directly and avoid the machinery of white noise calculus. Our arguments combine this probabilistic input with analytic techniques and may be of independent interest. Beyond the identity for the absolute value function, our method also extends to a class of convex functions, yielding a generalized Tanaka-type formula in that setting. 
	
	The subsequent sections are organized as follows. In Section 2, we recall the necessary background on fractional Brownian motion, Skorokhod integrals, and Malliavin calculus. Section 3 contains the precise framework and formulation of the main result, alongside the consideration of some important examples of interest. An extension to general convex functions is discussed as well. The proofs for the main result are given in the following sections. We conclude the paper with a separate discussion of a pathwise version of the Tanaka formula, based on Riemann--Stieltjes integration, and illustrate its applicability to equations with irregular drift.

	\section{Preliminaries}
	
	In this section we briefly recall basic facts on fractional Brownian motion and stochastic calculus with respect to it. For details on the topic, we refer to \cite{NualartBook}.
	Denote by $R(s,t)$, $s,t \in [0,T],$ the covariance function of $B^H$, i.e.
	\[
	R(t,s):= \frac12 \Big(t^{2H} + s^{2H} - |t-s|^{2H}\Big).
	\]
	By extending the mapping $\mathbf{1}_{[0,t)} \mapsto B^H_t$ linearly and closing with respect to the inner product
	$$
	\langle \mathbf{1}_{[0,t)}, \mathbf{1}_{[0,s)}\rangle_{\mathcal{H}} = R(t,s)
	$$
	one obtains a Hilbert space $\mathcal{H}$ and an associated isonormal Gaussian process $\{B^H(\varphi):\varphi \in \mathcal{H}\}$ with the property
	$$
	\mathbb{E}[B^H(\varphi_1),B^H(\varphi_2)] = \langle \varphi_1,\varphi_2\rangle_{\mathcal{H}}.
	$$
	For $H>\frac12$, denote $|\mathcal{H}|$ the space of functions $\varphi$ that satisfy 
	$$
	\int_0^T \int_0^T |\varphi(t)||\varphi(s)||t-s|^{2H-2}dsdt < \infty.
	$$
	It is known that then $|\mathcal{H}| \subset \mathcal{H}$ (see \cite{pipiras2001classes}), and for $\varphi_1,\varphi_2 \in |\mathcal{H}|$ we have 
	$$
	\mathbb{E}[B^H(\varphi_1),B^H(\varphi_2)] = \alpha_H \int_0^T \int_0^T \varphi_1(t)\varphi_2(s)|t-s|^{2H-2}dsdt
	$$
	with $\alpha_H = H(2H-1)$. For our purposes it is sufficient to consider the space $|\mathcal{H}|$.

	We next review basic facts on Malliavin calculus and Skorokhod integration with respect to $B^H$. For a random variable $F$ of the form
	$$
	F = f\big( B(h_1), \ldots, B(h_n) \big)
	$$ 
	with $f \in C^\infty_b$ and $h_i \in \mathcal{H}$ ($f$ and all its partial derivatives are bounded), the Malliavin derivative is defined as 
	$$
	D F  = \sum_{i=1}^n  \frac{\partial f}{\partial x_i} \big( B(h_1), \ldots, B(h_n) \big)h_i.
	$$ 
	In particular, if $h_i \in |\mathcal{H}|$ for all $i=1,\ldots,n$, the derivative $DF$ exists as a process $(D_tF)_{t\in[0,T]}$ and we have 
	$$
	D_t F = \sum_{i=1}^n  \frac{\partial f}{\partial x_i} \big( B(h_1), \ldots, B(h_n) \big)h_i(t).
	$$
	The Sobolev space $D^{1,2}$ of square integrable Malliavin differentiable random variables is then defined by closing with respect to the norm
	$$
	\Vert F \Vert_{1,2} = \mathbb{E} F^2 + \mathbb{E}[\Vert DF\Vert^2_\mathcal{H}].
	$$
	The divergence operator $\delta$ as the adjoint of $D$ is defined as follows. For elements $u\in L^2(\Omega,\mathcal{H})$ we denote $u \in Dom(\delta)$ if 
	$$
	|\mathbb{E}\langle DF,u\rangle_{\mathcal{H}}| \leq c_u\Vert F\Vert_2
	$$
	for any $F\in D^{1,2}$, and in this case $\delta(u)$ is defined through the duality formula
	$$
	\mathbb{E}[F\delta(u)] = \mathbb{E}\langle DF,u\rangle_{\mathcal{H}}.
	$$
	The random variable \( \delta(u) \) is also called the Skorokhod integral and is denoted by \( \delta(u) = \int_0^T u \, \delta B^H \), as it coincides with the stochastic integral introduced by Skorokhod in the case of standard Brownian motion.

	\section{Main result: Tanaka formula for SDEs driven by fractional Brownian motion}

	\subsection*{Framework and Assumptions}
	
	We assume that $(X_t)_{t \in [0,T]}$ satisfies the stochastic differential equation
	\begin{align} \label{SDE10}
		\begin{split}
			dX_t & = b(X_t) dt + \sigma(X_t) dB^H_t , \quad t \in [0,T], \\
			X_0 & \in L^2,
		\end{split}
	\end{align}
	where $b, \sigma\colon \mathbb{R} \to \mathbb{R}$ are measurable functions such that the above expression is well-defined.%, and, moreover, $(X_t)_{t \in [0,T]}$ is H\"older continuous of every order $\gamma \in (0,H)$.
	
	We break down the main assumptions of the theorem into two conditions that are key properties used in the proof. In order to state and prove our main result, we impose the following:
	
	\begin{enumerate}[label=\textbf{(A\arabic*)}]
		%   \item \label{ass:bounded} There exists a constant $C > 0$ such that
		%\[
		% |\sigma(x)| \leq C \quad \text{for all } x \in \mathbb{R}.
		%\]
		%(That is, $\sigma$ is bounded.)

		\item \label{ass:integrability} For each $a>0$, $\gamma \in (0,H)$ and $a\leq t_1<t_2 < t_3< t_4 \leq T$ the law of $(X_{t_1}, X_{t_2} -X_{t_1} , X_{t_3} -X_{t_2} , X_{t_4} -X_{t_3})$ admits a smooth density $q_{t_1, t_2, t_3, t_4}$ satisying 
		$$ \partial^{\alpha_1}_{x_1 } \partial^{\alpha_1}_{x_2 } \partial^{\alpha_1}_{x_3 } \partial^{\alpha_1}_{x_4 } q_{s_1, s_2, s_3, s_4} ( x_1, x_2, x_3, x_4) \leq  C_1 \frac{ e^{-\frac{|x_1|^{2\gamma}}{|s_1|^{2\gamma^2}}}}{s_1^{(1+\alpha_1)H}}  \prod_{j=2}^4 \frac{ e^{-\frac{|x_j|^{2\gamma}}{C_2|s_j - s_{j-1}|^{2\gamma^2}}}}{(s_j - s_{j-1})^{(1 + \alpha_j)H}} 
		$$
		for every multi-index $\alpha=(\alpha_1, \alpha_2, \alpha_3, \alpha_4)$ and some constants $C_1, C_2 \in (0, \infty)$ only depending on $a$. Furthermore, if $\alpha_1 =\alpha_2 = \alpha_3 = \alpha_4= 0$, the corresponding inequality holds for every $0<t_1<t_2 < t_3< t_4 \leq T$, independent of $a$.
		%The following integral condition holds:
		%\[
		%\int_0^T \int_0^T \mathbb{E} \big[ |\sigma(X_u) \sigma(X_s)| |X_s  X_u|\big] |u-s|^{2H-2} \, du \, ds < \infty.
		%\mathbbm{1}_{\{|X_s  X_u| \geq 1\}}
		%\]
		\item \label{ass:malliavin} For every $t \in [0,T]$, $X_t$ is Malliavin differentiable (with respect to $B^H$), and the random variables
		\begin{align*}
			& \int_0^t b(X_s) \, ds, \quad \,\int_0^t \int_0^t  |\sigma'(X_r)| |D_r X_s|  |s-r|^{2H-2} \, dr \, ds, \\
			& \quad \int_0^t \int_0^t \sigma(X_s)  \sigma(X_u) D_rX_sD_pX_u |s-r|^{2H-2} |u-p|^{2H-2}\, du \, ds 
		\end{align*}
		belong to $L^2$, for every fixed $p,r,t \in [0,T]$.
	\end{enumerate}
	
	Assumption~\ref{ass:integrability} provides smoothness and Gaussian-type density estimates for the law of certain increments of the solution. This is a key property in our proofs, which is used to establish \( L^2 \)-convergence of the Malliavin trace term. Note that in Assumption~\ref{ass:integrability}, introduction of the parameter $\gamma$ allows for some flexibility on the decay of the tails, as sub-Gaussian decay is sufficient for our purposes. The bounds of   Assumption~\ref{ass:integrability} are similar to the ones presented in \cite{lou}, see also \cite{baudoin} for related results.
	
	Assumption~\ref{ass:malliavin} ensures the integrability of the correction term linking the Skorokhod and Riemann--Stieltjes integrals. All assumptions are satisfied by the examples treated below, which include cases not covered by a single SDE framework.
	
	The Tanaka formula for SDEs driven by fractional Brownian motion with $H>\frac12$ reads as follows. 
	
	\begin{theorem}\label{main}
		Let $(X_t)_{t \in [0,T]}$ satisfy \eqref{SDE10}. Assume that \ref{ass:integrability}, \ref{ass:malliavin} hold. Then, with probability one, for every $(t,x) \in [0,T] \times \rr$
		\begin{align} \label{TF}
			\begin{split}
				| X_t -x| & = | X_0 -x| + \int_0^t \operatorname{sgn} (X_s-x) b(X_s) ds + \int_0^t \operatorname{sgn} (X_s-x) \sigma(X_s) \delta B_s^H \\
				& \quad+ H(2H-1) \int_0^t \int_0^t [ \operatorname{sgn}(X_s-x)  \sigma'(X_r)D_r X_s] |s-r|^{2H-2} dr ds \\
				& \quad+ 2H(2H-1) \int_0^t \int_0^t [ \delta_x(X_r)\sigma(X_s)D_r X_s  ] |s-r|^{2H-2} dr ds .
			\end{split}
		\end{align}
		
	\end{theorem}
	\begin{remark}
		By carefully examining our proof, we actually obtain that for suitable convex functions we have 
		\begin{align*} 
			\begin{split}
				f(X_t) & = f(X_0) + \int_0^t f'_-(X_s) b(X_s) ds + \int_0^t f'_-(X_s) \sigma(X_s) \delta B_s^H \\
				& \quad+ H(2H-1) \int_0^t \int_0^t [ f'_-(X_s)  \sigma'(X_r)D_r X_s] |s-r|^{2H-2} dr ds \\
				& \quad+ 2 H(2H-1) \int_{\mathbb{R}}\int_0^t \int_0^t [ \delta_x(X_r)\sigma(X_s)D_r X_s  ] |s-r|^{2H-2} dr ds f''(dx),
			\end{split}
		\end{align*}
		where $f'_-$ denotes the left-derivative of $f$ and $f''$ denotes the Radon measure associated to the second derivative of $f$ (in the sense of distributions). Indeed, this follows directly from the representation
		$$
		f(x) = \alpha +\beta x + \int |x-a|f''(da),
		$$
		valid whenever $f''$ is compactly supported. This allows to generalize immediately to convex functions for which $f''$ is compactly supported. More general cases then follow by approximation arguments. However, then one needs to study the decay of mollified approximations $f''_n(X_s)$ of $\delta_x(X_s)$ in terms of level $x$ in order to obtain $L^2$-convergence, leading to integrability conditions related to $f''(dx)$. We leave the detailed study of this extension to the reader. 
	\end{remark}
	
	Before giving the proof we want to illustrate that our general Theorem \ref{main} includes some important examples.
	
	\subsection*{Examples}
	We illustrate the applicability of Theorem \ref{main} in several examples. The first two examples recover the known results in the cases of the fractional  Brownian motion and the fractional Ornstein-Uhlenbeck process. The third example involves equations with non-Gaussian solutions $X$.

	\subsubsection*{Fractional Brownian Motion}
	
	Consider the case where $b \equiv 0$ and $\sigma \equiv 1$, so that $X_t = B_t^H$ is a fractional Brownian motion. In this case, Assumptions \ref{ass:integrability}, \ref{ass:malliavin} clearly hold. Then formula \eqref{TF} simplifies to
	\begin{align*}
		| B^H_t -x| & = |x| +  \int_0^t \operatorname{sgn} (B^H_s-x)  \delta B_s^H \\
		& \quad + 2H (2H-1) \int_0^t \delta_x (B^H_s) \int_0^t D_r [B^H_s  ] |r-s|^{2H-2} dr ds.
	\end{align*}
	Using standard properties of the Malliavin derivative of fBM, we obtain
	\begin{align*}
		& 2H (2H-1) \int_0^t \delta_x (B^H_s) \int_0^s  (s-r)^{2H-2} dr ds = 2H  \int_0^t \delta_x (B^H_s) s^{2H-1} ds.
	\end{align*}
	This precisely recovers the statement of \cite[Theorem 3]{coutin}. 
	
	\subsubsection*{Fractional Ornstein-Uhlenbeck Process}
	
	Consider the fractional Ornstein-Uhlenbeck process given by $b(x) = -x$ and $\sigma \equiv \nu \in \mathbb{R}$ above, which is a Gaussian process, and Assumptions \ref{ass:integrability} and \ref{ass:malliavin} are satisfied. %Let us verify \ref{ass:integrability}. Since \( \sigma \) is a constant function, i.e., \( \sigma(x) \equiv \nu \), Assumptionreduces to verifying the finiteness of the integral
	%\[
	%\int_0^T \int_0^T \mathbb{E} \big[ |\nu^2 X_s X_u| \big] |u-s|^{2H-2} \, du \, ds.
	%\]
	%Factoring out \( \nu^2 \), we need to show that
	%\[
	%\nu^2 \int_0^T \int_0^T \mathbb{E} \big[ |X_s X_u| \big] |u-s|^{2H-2} \, du \, ds < \infty.
	%\]
	%Using the Cauchy-Schwarz inequality, we estimate
	%\[
	%\mathbb{E} \big[ |X_s X_u| \big] \leq \big( \mathbb{E} [X_s^2] \big)^{1/2} \big( \mathbb{E} [X_u^2] \big)^{1/2}.
	%\]
	%	Thus, it suffices to show that
	%	\[
	%	\nu^2 \int_0^T \int_0^T \big( \mathbb{E} [X_s^2] \mathbb{E} [X_u^2] \big)^{1/2} |u-s|^{2H-2} \, du \, ds < \infty.
	%	\]
	%	For the stationary fractional Ornstein-Uhlenbeck process, it is well-known that the variance is given by
	%	\[
	%	\mathbb{E} [X_t^2] = \frac{\nu^2}{2\lambda} \Gamma(2H-1).
	%	\]
	%	Since this expression is independent of \( t \), we obtain
	%	\[
	%	\mathbb{E} [X_s^2] = \mathbb{E} [X_u^2] = \frac{\nu^2}{2\lambda} \Gamma(2H-1),
	%	\]
	%	so that
	%	\[
	%%	\big( \mathbb{E} [X_s^2] \mathbb{E} [X_u^2] \big)^{1/2} = \frac{\nu^2}{2\lambda} \Gamma(2H-1).
	%\]
	%Thus, the integral reduces to
	%	\[
	%	\frac{\nu^4}{4\lambda^2} \Gamma(2H-1) \int_0^T \int_0^T |u-s|^{2H-2} \, du \, ds,
	%	\]
	%	which is finite , since \( H <1\). Thus, Assumption \ref{ass:integrability} also holds for the fractional Ornstein-Uhlenbeck process. 
	Then \eqref{TF} becomes
	\begin{align*}
		| X_t -x| & = | X_0 -x| - \int_0^t \operatorname{sgn} (X_s-x) X_s ds + \nu\int_0^t \operatorname{sgn} (X_s-x)  \delta B_s^H \\
		& \quad + 2\nu H (2H-1) \int_0^t \delta_x (X_s) \int_0^t D_r [X_s ] |r-s|^{2H-2} dr ds.
	\end{align*}
	Similarly to \cite[Proposition 7]{NS}, we have $D_r [X_s ] = 0$ for every $r \geq s$, so the last term simplifies to
	\[
	2\nu H \int_0^t \delta_x (X_s) \int_0^\infty  (2H-1) D_r [X_s ] (s-r)^{2H-2} dr ds.
	\]
	This coincides with \cite[Corollary 3.5]{yantian}.

	\subsubsection*{Equations with non-Gaussian solutions}
	
	Now consider the stochastic differential equation
	\begin{align} \label{SDE12}
		\begin{split}
			dX_t & = \sigma(X_t) dB^H_t , \quad t \in [0,T], \\
			X_0 & \in \rr,
		\end{split}
	\end{align}
	with $\sigma  \in C_b^\infty$ bounded away from zero, where $C_b^\infty$ denotes the class of infinitely differentiable functions on $\rr$ with bounded partial derivatives of all orders. Following the arguments of \cite[Section 2]{torres} we know that a solution to \eqref{SDE12} is given by $X_t = \Lambda^{-1}(B_t+\Lambda(X_0))$, where 
	$$
	\Lambda(x)=\int_0^x \frac{1}{\sigma(y)}dy.
	$$
	Using $\sigma \in C_b^\infty$, it is straightforward to check that now $X$ has Gaussian-type density from which Assumption~\ref{ass:integrability} follows from straightforward computations. For Assumption~\ref{ass:malliavin}, note that now 
	$D_sX_t = \sigma(\Lambda(B_t+\Lambda(X_0)))$ for $s< t$ and $D_sX_t=0$ otherwise. Thus $D_sX_t$ is uniformly bounded which gives Assumption~\ref{ass:malliavin}.

	%, we use \cite[Theorem 1.5]{baudoin} and recall the fact that $X_s$, $s>0$, admits a density $p_s$ satisfying 
	%$$
	%p_s(y) \leq Cs^{-H}e^{-C(y-x_0)^2/2s^{2H}}
	%$$
	%for every $y\in \rr$, where $C>0$ denotes some constant independent of $y$. Then it holds
	%\begin{align*}
	%	\mathbb{E} X_s^2 &= \int_{-\infty}^{\infty} y^2 p_s(y)dy\leq C\int_{-\infty}^{\infty} y^2 s^{-H}e^{-C(y-x_0)^2/2s^{2H}} dy \\
	%	&\leq C s^{2H}
	%\end{align*}
	%	for some constant $C>0$.Thus, by Cauchy-Schwarz inequality, along with the fact that $\sigma$ is bounded, we get
	%	\begin{align*}
		%		& \int_0^T \int_0^T \mathbb{E} \big[ |\sigma(X_u) \sigma(X_s)| |X_s  X_u|\big] |u-s|^{2H-2} \, du \, ds \\
		%		& \leq C \int_0^t\int_0^t \mathbb{E}[|X_u|^2]^{\frac12}\mathbb{E}[|X_s|^2]^{\frac12}|u-s|^{2H-2}duds \\
		%		& \leq C \int_0^t \int_0^t  u^{H}s^{H}|u-s|^{2H-2} duds ,
		%	\end{align*} 
	%	which is finite since $H>\frac12$. 
	
	\section{Proof of Theorem \ref{main}}

	\subsection{Mollification}\label{mol}
	
	For notational simplicity and without using generality, we will prove Theorem \ref{main} with $x=0$. We provide several results on the convergence of various terms that eventually will lead to our main result. In the following we denote by
	$$ f(z) = |z| , \quad z \in \rr,$$
	and by
	$$\rho_\varepsilon(z) :=  \frac{1}{\sqrt{2\pi\varepsilon}} \exp \Big( -\frac{z^2}{2\varepsilon} \Big)  $$
	the Gaussian kernel, which satisfies $\int_{- \infty}^\infty \rho_\varepsilon(y) dy = 1$ for all $\varepsilon >0$.

	Consider the following convex, smooth approximations to $f$. Let the sequence
	$$f'_n(z) := 2\int_{- \infty}^z \rho_{\frac{1}{n}} (y)  dy -1, \quad z \in \rr, \quad n \in \N,$$
	and
	\begin{equation} \label{fnapp}
		f_n(z) := \int_{0}^z f_n'(y)  dy, \quad z \in \rr, \quad n \in \N.
	\end{equation}
	Then $f'_n(z)$ converges to $\operatorname{sgn}(z)$, and $f_n(z)$ converges to $f(z)$, as $n \to \infty$.
	
	From the change of variables for fractional integrals \cite[Theorem 4.3.1]{za} we get
	\begin{equation} \label{tchange}
		f_n (X_t)  -  f_n(X_0) =  \int_0^t  f_n'(X_s) b(X_s)  ds + \int_0^t f_n'(X_s) \sigma(X_s) dB^H_s,
	\end{equation} 
	The sequence $f_n$, $n \in \N,$ approximates the function $f$, and we aim at taking limits on both sides, in order to obtain an expression for $|X_t|$. The next step is to show that we can rewrite \cite{alos1,alos2}
	\begin{align*}
		&\int_0^t f'_n(X_s) \sigma(X_s) dB^H_s\\
		&  = \int_0^t f'_n(X_s) \sigma(X_s) \delta B^H_s + H(2H-1) \int_0^t \int_0^t D_r [f'_n(X_s) \sigma(X_s)] |s-r|^{2H-2} dr ds,
	\end{align*}
	with
	\begin{align*}
		& H(2H-1) \int_0^t \int_0^t D_r [f'_n(X_s) \sigma(X_s)] |s-r|^{2H-2} dr ds\\
		& =   H(2H-1) \int_0^t \int_0^t [ \sigma(X_s) f''_n(X_r)D_r X_s ] |s-r|^{2H-2} dr ds \\
		& \quad+ H(2H-1) \int_0^t \int_0^t [f'_n(X_s)  \sigma'(X_r)D_r X_s] |s-r|^{2H-2} dr ds.
	\end{align*}
	Formally, one has to justify this by proving that the latter terms are almost surely finite. This will be the subject of the next section, where we in fact prove that the terms converge in $L^2$.

	\subsection{On the Malliavin trace term}
	
	In this section we aim at showing the $L^2$-convergence of the trace term. We begin with the following statement.
	
	\begin{lemma}\label{trace1}
		The sequence
		\[
		\left( \int_0^T f_n''(X_s) ds \right) , \quad n\in \mathbb{N},
		\]
		is Cauchy in $L^4$.
	\end{lemma}
	
	\begin{proof}
		Let us fix $n,m \in \mathbb{N}$ and consider
		\begin{align*}
			&\left( \int_0^T f_n''(X_s) ds - \int_0^T f_m''(X_s) ds \right)^4 \\
			&= \left( \int_0^T f_n''(X_s) ds \right)^4 - 4 \left( \int_0^T f_n''(X_s) ds \right)^3 \left( \int_0^T f_m''(X_s) ds \right) \\
			& \quad + 6 \left( \int_0^T f_n''(X_s) ds \right)^2 \left( \int_0^T f_m''(X_s) ds \right)^2 \\
			& \quad- 4 \left( \int_0^T f_n''(X_s) ds \right) \left( \int_0^T f_m''(X_s) ds \right)^3 \\
			& \quad+ \left( \int_0^T f_m''(X_s) ds \right)^4.
		\end{align*}
		
		Now taking expectation of each term and interchanging integrals with expectations we get
		
		\begin{align*}
			&\mathbb{E} \left[ \left( \int_0^T f_n''(X_s) ds - \int_0^T f_m''(X_s) ds \right)^4 \right] \\
			&= \mathbb{E} \left[ \left( \int_0^T f_n''(X_s) ds \right)^4 \right]
			- 4\, \mathbb{E} \left[ \left( \int_0^T f_n''(X_s) ds \right)^3 \left( \int_0^T f_m''(X_s) ds \right) \right] \\
			&\quad + 6\, \mathbb{E} \left[ \left( \int_0^T f_n''(X_s) ds \right)^2 \left( \int_0^T f_m''(X_s) ds \right)^2 \right]
			- 4\, \mathbb{E} \left[ \left( \int_0^T f_n''(X_s) ds \right) \left( \int_0^T f_m''(X_s) ds \right)^3 \right] \\
			&\quad + \mathbb{E} \left[ \left( \int_0^T f_m''(X_s) ds \right)^4 \right] \\
			&= \int_{[0,T]^4} \mathbb{E} \left[ f_n''(X_{s_1}) f_n''(X_{s_2}) f_n''(X_{s_3}) f_n''(X_{s_4}) \right] ds_1 ds_2 ds_3 ds_4 \\
			&\quad - 4 \int_{[0,T]^4} \mathbb{E} \left[ f_n''(X_{s_1}) f_n''(X_{s_2}) f_n''(X_{s_3}) f_m''(X_{s_4}) \right] ds_1 ds_2 ds_3 ds_4 \\
			&\quad + 6 \int_{[0,T]^4} \mathbb{E} \left[ f_n''(X_{s_1}) f_n''(X_{s_2}) f_m''(X_{s_3}) f_m''(X_{s_4}) \right] ds_1 ds_2 ds_3 ds_4 \\
			&\quad - 4 \int_{[0,T]^4} \mathbb{E} \left[ f_n''(X_{s_1}) f_m''(X_{s_2}) f_m''(X_{s_3}) f_m''(X_{s_4}) \right] ds_1 ds_2 ds_3 ds_4 \\
			&\quad + \int_{[0,T]^4} \mathbb{E} \left[ f_m''(X_{s_1}) f_m''(X_{s_2}) f_m''(X_{s_3}) f_m''(X_{s_4}) \right] ds_1 ds_2 ds_3 ds_4 \\
			& =  \int_{[0,T]^4} \int_{\mathbb{R}^{4}} f_n''(x_1) f_n''(x_2) f_n''(x_3) f_n''(x_4) \, p_{s_1, s_2, s_3, s_4}(x_1, x_2, x_3, x_4) \, dx_1 dx_2 dx_3 dx_4 \, ds_1 ds_2 ds_3 ds_4 \\
			&\quad - 4 \int_{[0,T]^4} \int_{\mathbb{R}^{4}} f_n''(x_1) f_n''(x_2) f_n''(x_3) f_m''(x_4) \, p_{s_1, s_2, s_3, s_4}(x_1, x_2, x_3, x_4) \, dx_1 dx_2 dx_3 dx_4 \, ds_1 ds_2 ds_3 ds_4 \\
			&\quad + 6 \int_{[0,T]^4} \int_{\mathbb{R}^{4}} f_n''(x_1) f_n''(x_2) f_m''(x_3) f_m''(x_4) \, p_{s_1, s_2, s_3, s_4}(x_1, x_2, x_3, x_4) \, dx_1 dx_2 dx_3 dx_4 \, ds_1 ds_2 ds_3 ds_4 \\
			&\quad - 4 \int_{[0,T]^4} \int_{\mathbb{R}^{4}} f_n''(x_1) f_m''(x_2) f_m''(x_3) f_m''(x_4) \, p_{s_1, s_2, s_3, s_4}(x_1, x_2, x_3, x_4) \, dx_1 dx_2 dx_3 dx_4 \, ds_1 ds_2 ds_3 ds_4 \\
			&\quad + \int_{[0,T]^4} \int_{\mathbb{R}^{4}} f_m''(x_1) f_m''(x_2) f_m''(x_3) f_m''(x_4) \, p_{s_1, s_2, s_3, s_4}(x_1, x_2, x_3, x_4) \, dx_1 dx_2 dx_3 dx_4 \, ds_1 ds_2 ds_3 ds_4 ,
		\end{align*}
		where $p_{s_1, s_2, s_3, s_4}$ denotes the 4-dimensional density of the law of $(X_{s_1}, X_{s_2}, X_{s_3}, X_{s_4})$. We know from Assumption \ref{ass:integrability} that $p_{s_1, s_2, s_3, s_4}$ is continuous and satisfies
		$$ p_{s_1, s_2, s_3, s_4} \leq C s_1^{-H} (s_2-s_1)^{-H} (s_3-s_2)^{-H} (s_4-s_3)^{-H},$$ 
		where $C$ is independent of the arguments. Now, using these facts, we first apply a change of variables in the above chain of equalities. Denoting $g(x) = C \exp (-\frac{x^2}{2})$, we get
		\begin{align*}
			&\mathbb{E} \left[ \left( \int_0^T f_n''(X_s) ds - \int_0^T f_m''(X_s) ds \right)^4 \right] \\
			& =  \int_{[0,T]^4} \int_{\mathbb{R}^{4}} g(x_1) g(x_2)g(x_3) g(x_4) \, p_{s_1, s_2, s_3, s_4}\Big(\frac{x_1}{\sqrt{n}}, \frac{x_2}{\sqrt{n}}, \frac{x_3}{\sqrt{n}}, \frac{x_4}{\sqrt{n}}\Big) \, dx_1 dx_2 dx_3 dx_4 \, ds_1 ds_2 ds_3 ds_4 \\
			&\quad - 4 \int_{[0,T]^4} \int_{\mathbb{R}^{4}} g(x_1) g(x_2)g(x_3) g(x_4) \, p_{s_1, s_2, s_3, s_4}\Big(\frac{x_1}{\sqrt{n}}, \frac{x_2}{\sqrt{n}}, \frac{x_3}{\sqrt{n}}, \frac{x_4}{\sqrt{m}}\Big)  \, dx_1 dx_2 dx_3 dx_4 \, ds_1 ds_2 ds_3 ds_4 \\
			&\quad + 6 \int_{[0,T]^4} \int_{\mathbb{R}^{4}}g(x_1) g(x_2)g(x_3) g(x_4) \, p_{s_1, s_2, s_3, s_4}\Big(\frac{x_1}{\sqrt{n}}, \frac{x_2}{\sqrt{n}}, \frac{x_3}{\sqrt{m}}, \frac{x_4}{\sqrt{m}}\Big)  \, dx_1 dx_2 dx_3 dx_4 \, ds_1 ds_2 ds_3 ds_4 \\
			&\quad - 4 \int_{[0,T]^4} \int_{\mathbb{R}^{4}} g(x_1) g(x_2)g(x_3) g(x_4) \, p_{s_1, s_2, s_3, s_4}\Big(\frac{x_1}{\sqrt{n}}, \frac{x_2}{\sqrt{m}}, \frac{x_3}{\sqrt{m}}, \frac{x_4}{\sqrt{m}}\Big)  \, dx_1 dx_2 dx_3 dx_4 \, ds_1 ds_2 ds_3 ds_4 \\
			&\quad + \int_{[0,T]^4} \int_{\mathbb{R}^{4}} g(x_1) g(x_2)g(x_3) g(x_4) \, p_{s_1, s_2, s_3, s_4}\Big(\frac{x_1}{\sqrt{m}}, \frac{x_2}{\sqrt{m}}, \frac{x_3}{\sqrt{m}}, \frac{x_4}{\sqrt{m}}\Big) \, dx_1 dx_2 dx_3 dx_4 \, ds_1 ds_2 ds_3 ds_4 .
		\end{align*}
		Now, using the continuity of the density, together with dominated convergence,
		we get that, as \(n, m \to \infty\), the latter expression vanishes.
	\end{proof}
	
	Having established Lemma \ref{trace1}, next we claim that the convergence
	\[
	\mathbb{E} \left[ \left( \int_0^t f_n''(X_s) ds - \int_0^t f_m''(X_s) ds \right)^4 \right]
	\]
	is uniform in $t \in [0,T]$.
	
	\begin{lemma}\label{trace2}
		It holds
		\[
		\lim_{n,m \to \infty}\sup_{t \in [0,T ]}\mathbb{E} \left[ \left( \int_0^t f_n''(X_s) ds - \int_0^t f_m''(X_s) ds \right)^4 \right] = 0.
		\]		
	\end{lemma}
	
	\begin{proof}
		In order to prove this, we will use the bounds in Assumption \ref{ass:integrability}. Let us first fix $a >0$ and estimate
		\begin{align*}
			&\sup_{t \in [0,T ]}\mathbb{E} \left[ \left( \int_0^t f_n''(X_s) ds - \int_0^t f_m''(X_s) ds \right)^4 \right]   \\
			& \leq c  \mathbb{E} \left[ \left( \int_0^a f_n''(X_s) ds - \int_0^a f_m''(X_s) ds \right)^4 \right] +c \sup_{t \in [0,T ]}\mathbb{E} \left[ \left( \int_a^t f_n''(X_s) ds - \int_a^t f_m''(X_s) ds \right)^4 \right] ,
		\end{align*}
		where the first term tends to zero by Lemma \ref{trace1}. Thus, it suffices to consider the second term. We will apply similar calculations as before. Recall in the following that we denote by  $q_{s_1, s_2, s_3, s_4}(x_1,x_2,x_3,x_4)$ the density of $(X_{s_1}, X_{s_2} - X_{s_1}, X_{s_3}-X_{s_2}, X_{s_4}-X_{s_3})$. Denoting $[a,T]^4_< := \{ (s_1, s_2, s_3, s_4) \in [a,T]^4 \mid s_1 < s_2 < s_3 < s_4 \}$, abbreviating $ \bar{x} = (x_1,x_2,x_3,x_4)$ and
		\begin{align*}g_n(\bar{x} ) & = f_n''(x_1) f_n''(x_2+x_1)f_n''(x_3+x_2+x_1) f_n''(x_4+x_3+x_2+x_1), \\
			g_{n,m,1}(\bar{x} ) & = f_n''(x_1) f_n''(x_2+x_1)f_n''(x_3+x_2+x_1) f_m''(x_4+x_3+x_2+x_1), \\
			g_{n,m,2}(\bar{x} ) & = f_n''(x_1) f_n''(x_2+x_1)f_m''(x_3+x_2+x_1) f_m''(x_4+x_3+x_2+x_1), \\
			g_{n,m,3}(\bar{x} ) & = f_n''(x_1) f_m''(x_2+x_1)f_m''(x_3+x_2+x_1) f_m''(x_4+x_3+x_2+x_1),
		\end{align*}
		then we have
		
		\begin{align*}
			& \frac{1}{24!}\mathbb{E} \left[ \left( \int_a^t f_n''(X_s) ds - \int_a^t f_m''(X_s) ds \right)^4 \right] \\
			& =  \int_{[a,t]^4_<} \int_{\mathbb{R}^{4}} g_n(\bar{x} ) \, q_{s_1, s_2, s_3, s_4}(x_1,x_2,x_3,x_4) \, dx_1 dx_2 dx_3 dx_4 \, ds_1 ds_2 ds_3 ds_4 \\
			&\quad - 4 \int_{[a,t]_<^4} \int_{\mathbb{R}^{4}} g_{n,m,1}(\bar{x} ) \, q_{s_1, s_2, s_3, s_4}(x_1,x_2,x_3,x_4)  \, dx_1 dx_2 dx_3 dx_4 \, ds_1 ds_2 ds_3 ds_4 \\
			&\quad + 6  \int_{[a,t]_<^4} \int_{\mathbb{R}^{4}} g_{n,m,2}(\bar{x} ) \, q_{s_1, s_2, s_3, s_4}(x_1,x_2,x_3,x_4)   \, dx_1 dx_2 dx_3 dx_4 \, ds_1 ds_2 ds_3 ds_4  \\
			&\quad - 4  \int_{[a,t]_<^4} \int_{\mathbb{R}^{4}}  g_{n,m,3}(\bar{x} ) \, q_{s_1, s_2, s_3, s_4}(x_1,x_2,x_3,x_4)   \, dx_1 dx_2 dx_3 dx_4 \, ds_1 ds_2 ds_3 ds_4  \\
			&\quad + \int_{[a,t]_<^4} \int_{\mathbb{R}^{4}} g_m(\bar{x} )  \, q_{s_1, s_2, s_3, s_4}(x_1,x_2,x_3,x_4)  \, dx_1 dx_2 dx_3 dx_4 \, ds_1 ds_2 ds_3 ds_4  .
		\end{align*}
		Now observing that 
		\begin{align*}
			& \int_{[a,t]^4_<} \int_{\mathbb{R}^{4}} g_n(\bar{x} ) \, q_{s_1, s_2, s_3, s_4}(0,0,0,0) \, dx_1 dx_2 dx_3 dx_4 \, ds_1 ds_2 ds_3 ds_4, \\
			&=  \int_{[a,t]_<^4} \int_{\mathbb{R}^{4}} g_{n,m,i}(\bar{x} ) \, q_{s_1, s_2, s_3, s_4}(0,0,0,0)  \, dx_1 dx_2 dx_3 dx_4 \, ds_1 ds_2 ds_3 ds_4 \in (0, \infty)
		\end{align*}
		for $ i=1,2,3$, independent of $n,m$, we obtain that
		
		\begin{align*}
			& \frac{1}{24!}\mathbb{E} \left[ \left( \int_a^t f_n''(X_s) ds - \int_a^t f_m''(X_s) ds \right)^4 \right] \\
			& =  \int_{[a,t]_<^4} \int_{\mathbb{R}^{4}} g_n(\bar{x} ) \, \left(q_{s_1, s_2, s_3, s_4}(x_1,x_2,x_3,x_4)  - q_{s_1, s_2, s_3, s_4}(0,0,0,0)\right) \, d\bar{x} \, ds_1 ds_2 ds_3 ds_4 \\
			&\quad + 4 \int_{[a,t]_<^4} \int_{\mathbb{R}^{4}} g_{n,m,1}(\bar{x} )  \left(  q_{s_1, s_2, s_3, s_4}(0,0,0,0) - q_{s_1, s_2, s_3, s_4}(x_1,x_2,x_3,x_4)  \right) \, d\bar{x} \, ds_1 ds_2 ds_3 ds_4 \\
			&\quad + 6 \int_{[a,t]_<^4} \int_{\mathbb{R}^{4}} g_{n,m,2}(\bar{x} )\, \left( q_{s_1, s_2, s_3, s_4}(x_1,x_2,x_3,x_4)  - q_{s_1, s_2, s_3, s_4}(0,0,0,0)\right) \, d\bar{x} \, ds_1 ds_2 ds_3 ds_4 \\
			&\quad + 4 \int_{[a,t]_<^4} \int_{\mathbb{R}^{4}} g_{n,m,3}(\bar{x} )  \left(  q_{s_1, s_2, s_3, s_4}(0,0,0,0) - q_{s_1, s_2, s_3, s_4}(x_1,x_2,x_3,x_4)  \right) \, d\bar{x} \, ds_1 ds_2 ds_3 ds_4  \\
			&\quad +  \int_{[a,t]_<^4} \int_{\mathbb{R}^{4}} g_m\bar{x}) \, \left( q_{s_1, s_2, s_3, s_4}(x_1,x_2,x_3,x_4) - q_{s_1, s_2, s_3, s_4}(0,0,0,0)\right) \, d\bar{x} \, ds_1 ds_2 ds_3 ds_4  .
		\end{align*}
		Next we use multivariate Taylor expansion and known upper bounds for the partial derivatives
		$$ \partial^{\alpha_1}_{x_1 } \partial^{\alpha_1}_{x_2 } \partial^{\alpha_1}_{x_3 } \partial^{\alpha_1}_{x_4 } q_{s_1, s_2, s_3, s_4} ( x_1, x_2, x_3, x_4) , $$
		where we denote by $\alpha = ({\alpha_1},{\alpha_2},{\alpha_3},{\alpha_4})$ a multi-index. Indeed, let us rewrite (in multivariate notation)
		\begin{align*}
			q_{s_1, s_2, s_3, s_4}(x_1,x_2,x_3,x_4) - q_{s_1, s_2, s_3, s_4}(0,0,0,0)\ & = \sum_{|\alpha| \geq 1} \frac{ \partial^{\alpha_1}_{x_1 } \partial^{\alpha_1}_{x_2 } \partial^{\alpha_1}_{x_3 } \partial^{\alpha_1}_{x_4 } q_{s_1, s_2, s_3, s_4} ( x_1, x_2, x_3, x_4)}{\alpha !} \bar{x}^\alpha .
		\end{align*}
		In the following let $C$ be a generic positive constant. By Assumption \ref{ass:integrability} for $\gamma=\gamma_\alpha <H$ we can estimate 
		$$ \partial^{\alpha_1}_{x_1 } \partial^{\alpha_1}_{x_2 } \partial^{\alpha_1}_{x_3 } \partial^{\alpha_1}_{x_4 } q_{s_1, s_2, s_3, s_4} ( x_1, x_2, x_3, x_4) \leq  C \frac{ e^{-\frac{|x_1|^{2\gamma}}{|s_1|^{2\gamma^2}}}}{s_1^{(1+\alpha_1)H}}  \prod_{j=2}^4 \frac{ e^{-\frac{|x_j|^{2\gamma}}{C|s_j - s_{j-1}|^{2\gamma^2}}}}{(s_j - s_{j-1})^{(1 + \alpha_j)H}} .
		$$
		Using these arguments we will now estimate
		$$  \int_{[a,t]_<^4} \int_{\mathbb{R}^{4}} g_n(\bar{x} ) \, \left(q_{s_1, s_2, s_3, s_4}(x_1,x_2,x_3,x_4)  - q_{s_1, s_2, s_3, s_4}(0,0,0,0)\right) \, d\bar{x} \, ds_1 ds_2 ds_3 ds_4.$$
		The other terms
		$$  \int_{[a,t]_<^4} \int_{\mathbb{R}^{4}} g_{n,m,i}(\bar{x} )  \left(  q_{s_1, s_2, s_3, s_4}(0,0,0,0) - q_{s_1, s_2, s_3, s_4}(x_1,x_2,x_3,x_4)  \right) \, d\bar{x} \, ds_1 ds_2 ds_3 ds_4 , \quad i=1,2,3,4,$$
		are estimated analogously. We have
		\begin{align*}
			& \int_{[a,t]_<^4} \int_{\mathbb{R}^{4}} g_n(\bar{x} ) \, \left(q_{s_1, s_2, s_3, s_4}(x_1,x_2,x_3,x_4)  - q_{s_1, s_2, s_3, s_4}(0,0,0,0)\right) \, d\bar{x} \, ds_1 ds_2 ds_3 ds_4 \\
			& = \sum_{|\alpha| \geq 1} \int_{[a,t]_<^4} \int_{\mathbb{R}^{4}} g_n(\bar{x} ) \, \frac{ \partial^{\alpha_1}_{x_1 } \partial^{\alpha_1}_{x_2 } \partial^{\alpha_1}_{x_3 } \partial^{\alpha_1}_{x_4 } q_{s_1, s_2, s_3, s_4} ( x_1, x_2, x_3, x_4)}{\alpha !}  \bar{x}^\alpha \, d\bar{x} \, ds_1 ds_2 ds_3 ds_4 \\
			& \leq  C \sum_{|\alpha| \geq 1} \int_{[a,t]_<^4} \int_{\mathbb{R}^{4}} g_n(\bar{x} )  \frac{1}{\alpha !}  \bar{x}^\alpha \, {  \frac{ e^{-\frac{|x_1|^{2\gamma}}{|s_1|^{2\gamma^2}}}}{s_1^{(1+\alpha_1)H}}  \prod_{j=2}^4 \frac{ e^{-\frac{|x_j|^{2\gamma}}{C|s_j - s_{j-1}|^{2\gamma^2}}}}{(s_j - s_{j-1})^{(1 + \alpha_j)H}} } \, d\bar{x} \, ds_1 ds_2 ds_3 ds_4\\
			& = C \sum_{\substack{|\alpha| \geq 1 \\ \alpha_j \neq 0 \forall j}}  \int_{[a,t]_<^4} \int_{\mathbb{R}^{4}} g_n(\bar{x} )  \frac{1}{\alpha !}  \bar{x}^\alpha \, {  \frac{ e^{-\frac{|x_1|^{2\gamma}}{|s_1|^{2\gamma^2}}}}{s_1^{(1+\alpha_1)H}}  \prod_{j=2}^4 \frac{ e^{-\frac{|x_j|^{2\gamma}}{C|s_j - s_{j-1}|^{2\gamma^2}}}}{(s_j - s_{j-1})^{(1 + \alpha_j)H}} } \, d\bar{x} \, ds_1 ds_2 ds_3 ds_4 \\
			& \quad +  C \sum_{\substack{|\alpha| \geq 1 \\ \exists j: \alpha_j= 0}}  \int_{[a,t]_<^4} \int_{\mathbb{R}^{4}} g_n(\bar{x} )  \frac{1}{\alpha !}  \bar{x}^\alpha \, {  \frac{ e^{-\frac{|x_1|^{2\gamma}}{|s_1|^{2\gamma^2}}}}{s_1^{(1+\alpha_1)H}}  \prod_{j=2}^4 \frac{ e^{-\frac{|x_j|^{2\gamma}}{C|s_j - s_{j-1}|^{2\gamma^2}}}}{(s_j - s_{j-1})^{(1 + \alpha_j)H}} } \, d\bar{x} \, ds_1 ds_2 ds_3 ds_4 .
		\end{align*}
		If $\alpha_j = 0$ for some $j$, we make the observation that the exponent $(1+ \alpha_j)H = H<1$ so that one can use obvious estimates and the following calculations to see that the latter sum tends to zero, uniformly in $t \in [0,T]$. Thus, let us consider the first term. Here we use the substitution $x_1 = s_1^\gamma y_1$, $x_i = C(s_i - s_{i-1})^\gamma y_i$, $i=2,3,4$, and abbreviate $ \bar{x} =  s C\bar{y}$, so that this  expression amounts to
		\begin{align*}
			& C  \sum_{\substack{|\alpha| \geq 1 \\ \alpha_j \neq 0 \forall j}} \int_{[a,t]_<^4} \int_{\mathbb{R}^{4}} g_n( s \bar{y})  \frac{1}{\alpha !}  \bar{y}^\alpha \, {  \frac{ e^{-{|y_1|^{2\gamma}}}}{s_1^{(1+\alpha_1)(H-\gamma)}}  \prod_{j=2}^4 \frac{ e^{-{|y_j|^{2\gamma}}}}{(s_j - s_{j-1})^{(1+ \alpha_j)(H-\gamma)}} } \, d\bar{y} \, ds_1 ds_2 ds_3 ds_4.
		\end{align*}
		Now, if $\alpha_j \geq 1$ for each $j$, we can choose a constant $C$ such that for sufficiently small $\varepsilon >0$
		$$ \prod_{j=1}^4 e^{-{|y_j|^{2\gamma}}} y_j^{\alpha_j} \leq C  \prod_{j=1}^4 |y_j|^\varepsilon.$$
		Using this inequality the previous expression can be further estimated by
		\begin{align*}
			& C  \sum_{\substack{|\alpha| \geq 1 \\ \alpha_j \neq 0 \forall j}} \frac{1}{\alpha !}\int_{[a,t]_<^4} {s_1^{-(1+\alpha_1)(H-\gamma)}} \prod_{j=2}^4 {(s_j - s_{j-1})^{-(1+ \alpha_j)(H-\gamma)}}   \int_{\mathbb{R}^{4}} g_n( s \bar{y})    \bar{y}^\varepsilon \,    d\bar{y} \, ds_1 ds_2 ds_3 ds_4.
		\end{align*}
		Now recalling the definition of $g_n(s \bar{y})$, using the substitution $y_1 = s_1^{\gamma} n^{-\frac12}z_1$, $y_i = (s_i - s_{i-1})^{\gamma} n^{-\frac12} z_i$, $i=2,3,4$, and observing that
		\begin{align*}
			\int_{\mathbb{R}^{4}} g_n( s \bar{y})    \bar{y}^\varepsilon \,    d\bar{y} = C n^{-2\varepsilon} {s_1^{-(\gamma+\varepsilon)}} \prod_{j=2}^4 {(s_j - s_{j-1})^{-(\gamma+\varepsilon)}} 
		\end{align*}
		we get that 
		\begin{align*}
			& C  \sum_{\substack{|\alpha| \geq 1 \\ \alpha_j \neq 0 \forall j}} \frac{1}{\alpha !}\int_{[a,t]_<^4} {s_1^{-(1+\alpha_1)(H-\gamma)}} \prod_{j=2}^4 {(s_j - s_{j-1})^{-(1+ \alpha_j)(H-\gamma)}}   \int_{\mathbb{R}^{4}} g_n( s \bar{y})    \bar{y}^\varepsilon \,    d\bar{y} \, ds_1 ds_2 ds_3 ds_4 \\
			& = C  n^{-2\varepsilon}\sum_{\substack{|\alpha| \geq 1 \\ \alpha_j \neq 0 \forall j}} \frac{1}{\alpha !}\int_{[a,t]_<^4} {s_1^{-(1+\alpha_1)H + \alpha_1 \gamma - \varepsilon}} \prod_{j=2}^4 {(s_j - s_{j-1})^{-(1+\alpha_j)H + \alpha_j \gamma - \varepsilon}}   \, ds_1 ds_2 ds_3 ds_4.
		\end{align*}
		Now choosing at each step $\gamma$ such that $\gamma \in ( \frac{1+\alpha_j}{\alpha_j} H - \frac{1-\varepsilon}{\alpha_j}, H)$, which is possible for sufficiently small $\varepsilon \in (0,1-H)$, since $H<1$, we see that the latter integrals are finite and bounded in $t\in [0,T]$. Overall, combining all the arguments, we get that
		\[
		\lim_{n,m \to \infty}\sup_{t \in [0,T ]}\mathbb{E} \left[ \left( \int_0^t f_n''(X_s) ds - \int_0^t f_m''(X_s) ds \right)^4 \right] = 0,
		\]
		as claimed.
	\end{proof}

	Now we proceed with the following statement.
	
	\begin{lemma}\label{trace3}
		The sequence
		$$
		\int_0^t \int_0^t f''_n(X_r)\sigma(X_s)D_rX_s|s-r|^{2H-2}dsdr, \quad n\in \mathbb{N},
		$$
		converges in $L^2$.
	\end{lemma}
	
	\begin{proof}
		The idea is to rewrite this expression as Riemann-Stieltjes integral as follows. Defining the process
		$$h_r :=  \int_0^t\sigma(X_s)D_rX_s|s-r|^{2H-2}\, ds$$
		we have
		\begin{align*}
			\int_0^t  f''_n(X_r)\int_0^t\sigma(X_s)D_rX_s|s-r|^{2H-2}dsdr & = \int_0^t h_r \,d \left(  \int_0^r f''_n(X_v) \, dv\right) \\
			&= \lim_{\| \mathcal{P} \| \to 0} \sum_{i=0}^{p-1} h_{p_i} \int_{p_i}^{p_{i+1}} f''_n(X_v) \, dv,
		\end{align*}
		where $\| \mathcal{P} \| $ denotes the length of the largest subinterval of the partition $\mathcal{P} $. Using this, along with Fatou's Lemma, the exchangeability of limits due to uniform convergence in Lemma \ref{trace2}, Cauchy-Schwarz inequality, and Assumption~\ref{ass:malliavin} we finally obtain
		
		\begin{align*}
			& \lim_{n,m \to \infty}\mathbb{E} \left[ \left( \int_0^t \int_0^t \left(f''_n(X_r) - f''_m(X_r)\right)\sigma(X_s)D_rX_s|s-r|^{2H-2}dsdr \right)^2\right] \\
			& =   \lim_{n,m \to \infty}\ \mathbb{E} \left[ \lim_{\| \mathcal{P} \| \to 0}  \lim_{\| \mathcal{Q} \| \to 0}  \sum_{i=0}^{p-1} h_{p_i}  \sum_{i=0}^{q-1} h_{q_i} \int_{p_i}^{p_{i+1}} \left(f''_n(X_w) - f''_m(X_w)\right) dw  \int_{q_i}^{q_{i+1}} \left(f''_n(X_v) - f''_m(X_v)\right) \, dv, \right] \\
			& \leq \lim_{n,m \to \infty}  \liminf_{\| \mathcal{P} \| \to 0}  \liminf_{\| \mathcal{Q} \| \to 0}  \sum_{i=0}^{p-1}  \sum_{i=0}^{q-1}\ \mathbb{E} \left[h_{p_i}  h_{q_i} \int_{p_i}^{p_{i+1}} \left(f''_n(X_w) - f''_m(X_w)\right) dw  \int_{q_i}^{q_{i+1}} \left(f''_n(X_v) - f''_m(X_v)\right) \, dv, \right] \\
			& \leq  \lim_{n,m \to \infty}  \liminf_{\| \mathcal{P} \| \to 0}  \liminf_{\| \mathcal{Q} \| \to 0}  \sum_{i=0}^{p-1}  \sum_{i=0}^{q-1}\  C \mathbb{E} \left[ \left| \int_{p_i}^{p_{i+1}} \left(f''_n(X_w) - f''_m(X_w)\right) dw \right|^4 \right]^{\frac{1}{4}} \mathbb{E} \left[ \left| \int_{p_i}^{p_{i+1}} \left(f''_n(X_v) - f''_m(X_v)\right) dv \right|^4 \right]^{\frac{1}{4}} \\
			& =   \liminf_{\| \mathcal{P} \| \to 0}  \liminf_{\| \mathcal{Q} \| \to 0}  \sum_{i=0}^{p-1}  \sum_{i=0}^{q-1}\   C \lim_{n,m \to \infty}\mathbb{E} \left[ \left| \int_{p_i}^{p_{i+1}} \left(f''_n(X_w) - f''_m(X_w)\right) dw \right|^4 \right]^{\frac{1}{4}} \mathbb{E} \left[ \left| \int_{p_i}^{p_{i+1}} \left(f''_n(X_v) - f''_m(X_v)\right) dv \right|^4 \right]^{\frac{1}{4}} \\
			&= 0,
		\end{align*}
		as claimed.
	\end{proof}

	\subsection{Conclusion}
	
	Let us summarize our previous findings. Now due to Assumption \ref{ass:malliavin} and Lemma \ref{trace3} we can rewrite \cite{alos1,alos2}
	\begin{align*}
		&\int_0^t f'_n(X_s) \sigma(X_s) dB^H_s\\
		&  = \int_0^t f'_n(X_s) \sigma(X_s) \delta B^H_s + H(2H-1) \int_0^t \int_0^t D_r [f'_n(X_s) \sigma(X_s)] |s-r|^{2H-2} dr ds,
	\end{align*}
	with
	\begin{align*}
		& H(2H-1) \int_0^t \int_0^t D_r [f'_n(X_s) \sigma(X_s)] |s-r|^{2H-2} dr ds\\
		& =   H(2H-1) \int_0^t \int_0^t [ \sigma(X_s) f''_n(X_r)D_r X_s ] |s-r|^{2H-2} dr ds \\
		& \quad+ H(2H-1) \int_0^t \int_0^t [f'_n(X_s)  \sigma'(X_r)D_r X_s] |s-r|^{2H-2} dr ds.
	\end{align*}
	Next, we can colclude that the latter expressions converge in $L^2$, using the results of the previous sections and dominated convergence theorem. By assumption $f'_n$ is bounded. Moreover, according to Assumption \ref{ass:malliavin} we have
	\begin{align*}
		\mathbb{E} \left[ \left( \int_0^t \int_0^t | \sigma'(X_r)| |D_rX_s| |s-r|^{2H-2} dr ds \right)^2\right]< \infty.
	\end{align*}
	Thus, the dominated convergence theorem applies to get
	\begin{align*}
		& H(2H-1) \lim_{n\to \infty}\int_0^t \int_0^t [f'_n(X_s)  \sigma'(X_r)D_r X_s] |s-r|^{2H-2} dr ds \\
		& = H(2H-1) \int_0^t \int_0^t [\operatorname{sgn}(X_s)  \sigma'(X_r)D_r X_s] |s-r|^{2H-2} dr ds,
	\end{align*}
	where the convergence is in $L^2$. Moreover, by Lemma \ref{trace3} the expression
	\begin{align*}
		&  H(2H-1) \int_0^t \int_0^t f''_n(X_r) \sigma(X_s) D_r X_s |s-r|^{2H-2} dr ds 
	\end{align*}
	also converges in $L^2$. We define its corresponding limit by the formal expression
	\begin{align*}
		2H(2H-1) \int_0^t \int_0^t \delta_0(X_r) \sigma(X_s) D_r X_s |s-r|^{2H-2} dr ds .
	\end{align*}
	
	Next let us recall the following lemma from \cite[Lemma 1]{coutin} regarding the fact that $\delta$ is a closable operator. 
	
	\begin{lemma}\label{close}
		Let $(u_n)$ be a sequence in $Dom(\delta)$. Assume that there is an $\mathcal{H}$-valued random variable $u$ such that
		\begin{itemize}
			\item $u_n$ converges to $u$ in $L^2(\Omega, \mathcal{H})$
			\item $\delta (u_n)$ converges in $L^2$ to some square integrable variable $G$.
		\end{itemize}
		Then it holds $u \in Dom(\delta)$ and $u=G$.
	\end{lemma}
	
	Clearly, the left-hand side of equation \eqref{tchange} converges in $L^2$. By assumption \ref{ass:malliavin}, this is also true for the first term in the right-hand side of \eqref{tchange}. Thus we conclude that the Skorokhod integral
	$$ \int_0^t f'_n(X_s) \sigma(X_s) \delta B^H_s$$
	also converges in $L^2$. Now combining all our findings and using Lemma \ref{close} we have obtained a proof for the following.
	
	\begin{prop}\label{Skorohodterm}
		We have $\operatorname{sgn}(X_s)\sigma(X_s)I_{(0,t)}(s) \in Dom(\delta)$ and 
		$$
		\lim_{n\to \infty}\int_0^T I_{(0,t)}(s)f'_n(X_s) \sigma(X_s) \delta B^H_s = \int_0^T I_{(0,t)}(s)\operatorname{sgn}(X_s) \sigma(X_s) \delta B^H_s,
		$$
		where the convergence is in $L^2$.
	\end{prop}
	We are now ready to complete the proof of Theorem \ref{main}.
	\begin{proof}[Proof of Theorem \ref{main}]
		Due to Proposition \ref{Skorohodterm} we get the following equation, as claimed in Theorem \ref{main}:
		\begin{align*}
			| X_t | & = | X_0 | + \int_0^t \operatorname{sgn} (X_s) b(X_s) ds + \int_0^t \operatorname{sgn} (X_s) \sigma(X_s) \delta B_s^H \\
			& \quad+ H(2H-1) \int_0^t \int_0^t [ \operatorname{sgn}(X_s)  \sigma'(X_r)D_r X_s] |s-r|^{2H-2} dr ds \\
			& \quad+  2H(2H-1) \int_0^t \int_0^t [ \delta_0(X_r)\sigma(X_s)D_r X_s  ] |s-r|^{2H-2} dr ds  \quad \text{ a.s.}
		\end{align*}
		This completes the proof.
	\end{proof}

	\section{A pathwise version of the Tanaka formula}
	We conclude this paper with a discussion of an alternative version of the Tanaka formula. In addition to the Skorokhod-based Tanaka formula established in Theorem~\ref{main}, one can also formulate a variant based entirely on pathwise integration. This alternative version relies on the fact that fractional Brownian motion with Hurst parameter \( H > \frac{1}{2} \) has zero quadratic variation. For completeness of the results, we also give the following Tanaka type formula related to pathwise integrals.
	
	\begin{theorem}\label{main2}
		Let $(X_t)_{t \in [0,T]}$ be given by \eqref{SDE10} with deterministic initial condition $X_0 =x_0 \in \rr$. Assume that \( \sigma \) is globally Lipschitz continuous and that $(X_t)_{t \in [0,T]}$ admits Hölder continuous paths of every order strictly less than $H$. Furthermore, assume that for all times \( t \in (0,T] \), the law of $X_t$ 
		admits a density \( p_{t}(x) \) satisfying $\sup_{x\in \mathbb{R}}p_t(x) \in L^1([0,T])$. Then, for every \( (t,x) \in [0,T] \times \mathbb{R} \), the following identity holds:
		\[
		|X_t - x| = |X_0 - x| + \int_0^t \operatorname{sgn}(X_s - x)\, b(X_s)\, ds + \int_0^t \operatorname{sgn}(X_s - x)\, \sigma(X_s)\, dB^H_s,
		\]
		where the integral with respect to \( B^H \) is understood in the Riemann--Stieltjes sense.
	\end{theorem}
	
	After mollification, the convergence of the drift term follows directly by using the defining SDE and the dominated convergence theorem. Thus the proof of Theorem~\ref{main2} follows from Proposition~\ref{pathterm}, stated and proved below.
	
	\begin{prop}\label{pathterm}
		Let $f$ be a convex function with right-derivative $f'$ and let $f_n$ be given by \eqref{fnapp}. Then
		$$ \lim_{n \to \infty} \int_0^T f_n'(X_s) \sigma(X_s) dB^H_s = \int_0^T f'(X_s) \sigma(X_s) dB^H_s$$
		almost surely for any $T<\infty$.
	\end{prop}
	\begin{proof}
		Following, e.g. \cite{chenlv,hinztv}, it suffices to show that, for some $\beta<H$,
		$$
		\Vert f'_n(X_\cdot)\sigma(X_\cdot) - f'(X_\cdot)\sigma(X_\cdot)\Vert_{2,\beta} \to 0,
		$$
		where
		$$
		\Vert g\Vert_{2,\beta} = \int_0^T \frac{|g(s)|}{s^\beta}ds + \int_0^T \int_0^T \frac{|g(t)-g(s)|}{|t-s|^{1+\beta}}dsdt.
		$$
		Since $\sigma$ is locally bounded and $\text{supp}(X)$ is compact by continuity of $X$, it follows that 
		\begin{equation*}
			\begin{split}
				&\int_0^T \frac{|f'_n(X_s)\sigma(X_s) - f'(X_s)\sigma(X_s)|}{s^\beta}ds \\
				&\leq C\int_0^T \frac{|f'_n(X_s)-f'(X_s)|}{s^\beta}ds\\
				&\leq C\Vert f'_n(X_\cdot)-f'(X_\cdot)\Vert_{2,\beta}.
			\end{split}
		\end{equation*}
		Similarly,
		\begin{equation*}
			\begin{split}
				&\int_0^T\int_0^T|f'_n(X_t)\sigma(X_t)-f'_n(X_s)\sigma(X_s) - f'(X_s)\sigma(X_s)+f'(X_t)\sigma(X_t)||t-s|^{-1-\beta}dsdt \\
				&\leq \int_0^T\int_0^T|\sigma(X_t)||f'_n(X_t)-f'(X_t)-f'_n(X_s)+f'(X_s)||t-s|^{-1-\beta}dsdt \\
				&+\int_0^T\int_0^T|\sigma(X_t)-\sigma(X_s)||f'_n(X_s) - f'(X_s)||t-s|^{-1-\beta}dsdt\\
				&\leq C\int_0^T\int_0^T|f'_n(X_t)-f'(X_t)-f'_n(X_s)+f'(X_s)||t-s|^{-1-\beta}dsdt \\
				&+\int_0^T\int_0^T|\sigma(X_t)-\sigma(X_s)||f'_n(X_s) - f'(X_s)||t-s|^{-1-\beta}dsdt\\
				&\leq C\int_0^T\int_0^T|f'_n(X_t)-f'(X_t)-f'_n(X_s)+f'(X_s)||t-s|^{-1-\beta}dsdt \\
				&+\Vert f'_n(X_\cdot)-f'(X_\cdot)\Vert_{2,\beta}.
			\end{split}
		\end{equation*}
		Now the solution $X$ is, in the terminology of \cite{hinztv}, $(s,1)$ variable for any $s\in (0,1)$ (see \cite[Example 3.12]{chenlv} and \cite[Example 3.9]{hinztv}). Hence, by \cite[Lemma 4.38]{hinztv} we have
		$$
		\Vert f'_n(X_\cdot)-f'(X_\cdot)\Vert_{2,\beta} \to 0
		$$
		almost surely and, consequently, it suffices to show 
		$$
		\lim_{n \to \infty}\int_0^T \int_0^T |\sigma(X_t)-\sigma(X_s)||f'_n(X_s) - f'(X_s)||t-s|^{-1-\beta}dsdt = 0.
		$$
		For this we use Lipschitz continuity of $\sigma$ and $(H-\varepsilon)$-H\"older continuity of $X$ to get
		\begin{equation*}
			\begin{split}
				&\int_0^T \int_0^T |\sigma(X_t)-\sigma(X_s)||f'_n(X_s) - f'(X_s)||t-s|^{-1-\beta}dtds \\
				&\leq C \int_0^T |f'_n(X_s)-f'(X_s)| \int_0^T |t-s|^{H-\varepsilon-\beta-1}dtds \\
				&\leq C \int_0^T |f'_n(X_s)-f'(X_s)|ds\\
				&\to 0,
			\end{split}
		\end{equation*}
		as $n \to \infty$. This concludes the proof.
	\end{proof}
	
	It is known that Tanaka-type formulas can be a useful analytical tool in the study of SDEs with irregular drift. An application of Theorem~\ref{main2} arises in such a setting. For instance, consider the example
	\[
	dX_t = b(X_t)\, dt + \sigma\, dB^H_t,
	\]
	where \( \sigma \in \mathbb{R} \setminus \{0\} \) is constant and \( b \in C^\alpha(\mathbb{R}) \) is a Hölder continuous function with some exponent \( \alpha > 1 - \frac{1}{2H} \). Under these assumptions, it was shown in \cite[Theorem 1.3]{li2023non} that the solution \( (X_t)_{t \in [0,T]} \) admits a density with Gaussian-type upper bound of the form
	\[
	p_t(x) \leq C_1 t^{-H} \exp\left( - C_2 \frac{(x - x_0)^2}{t^{2H}} \right).
	\]
	This ensures that the assumptions of Theorem~\ref{main2} are satisfied. In particular, our version of the pathwise Tanaka formula holds for this class of equations despite the lack of Lipschitz regularity in the drift, and reads explicitly as
	$$
	|X_t - x| = |X_0 - x| + \int_0^t \operatorname{sgn}(X_s - x)\, b(X_s)\, ds + \sigma \int_0^t \operatorname{sgn}(X_s - x)\, \, dB^H_s.$$
	
	Finally, we would like to remark that the pathwise Tanaka formula established in this section differs fundamentally from the Skorokhod-based version discussed earlier. Most notably, it does not feature a local time or correction term. This reflects the fact that fractional Brownian motion with \( H > \frac{1}{2} \) has zero quadratic variation, so the Riemann--Stieltjes integral suffices to describe the dynamics without additional trace terms. While the Skorokhod-based formula captures finer structural information through a Malliavin trace term, the pathwise version applies under a different set of assumptions and can be more accessible in certain irregular settings. The absence of a local time analogue highlights a fundamental difference in the nature of these two formulations.

	\bibliographystyle{amsplain}
	\bibliography{lit}
	
\end{document}